\documentclass[12pt]{article}
\usepackage{amsmath,amsfonts,amssymb,amsthm,latexsym}
\newtheorem{thm}{Theorem}
\newtheorem{defn}[thm]{Definition}

\newtheorem{cor}[thm]{Corollary}
\newtheorem{exam}[thm]{Example}
\newtheorem{prop}[thm]{Proposition}

\DeclareMathOperator{\fix}{Fix}
\begin{document}

\centerline {\Large\textbf{Fixed points of continuous rotative mappings}}
\centerline {\Large\textbf{ on the real line}}
\vskip.8cm \centerline {\textbf{Tammatada Pongsriiam}$^a$ and \textbf{Imchit Termwuttipong}$^a$\footnote{ Corresponding author}}

\vskip.5cm
\centerline {$^a$Department of Mathematics and Computer Science, Faculty of Science}
\centerline {Chulalongkorn University, Bangkok, 10330, Thailand}
\centerline {{\tt tammatada@gmail.com}}
\centerline {{\tt  Imchit.T@chula.ac.th}}

\begin{abstract}
We show that every continuous rotative mapping on a closed interval has a fixed point. This gives an answer to some open questions raised by Goebel and Koter.
\end{abstract}
\section{Introduction and Preliminaries}
In order to assure the existence of fixed points for nonexpansive mappings on Banach spaces, we need to impose some conditions on the space or on the mapping. Rotativeness is a property of mappings which assures the existence of fixed points in the case of nonexpansive mappings and in the case of $k$-Lipschitzian mappings provided that $k>1$ is not too large. The purpose of this article is to give a solution to some open questions on rotative mappings raised by Goebel and Koter \cite{GK}. So we adopt all their definitions with a little modification on their notation.
\begin{defn}
Let $C$ be a nonempty closed convex subset of a Banach space $X$, $T:C\to C$, $k>0$, $n\geq 2$, and $a\in [0, n)$.
\begin{itemize}
\item[(i)] A mapping $T$ is called \emph{$k$-Lipschitzian} if 
$$
\left\|Tx-Ty\right\| \leq k\left\|x-y\right\|
$$ 
for all $x, y\in C$. A mapping $T$ is called \emph{Lipschitzian} if it is $k$-Lipschitzian for some $k>0$. $T$ is called \emph{nonexpansive} if it is $1$-Lipschitzian, and $T$ is called \emph{contraction} if it is $k$-Lipschitzian for some $k<1$.  
\item[(ii)] A mapping $T$ is called \emph{$n$-periodic} if $T^n = I$, the identity map.
\item[(iii)] A mapping $T$ is called \emph{$(n,a)$-rotative} if 
$$
\left\|T^nx-x\right\| \leq a\left\|Tx-x\right\|
$$
for all $x\in C$. A mapping $T$ is called \emph{$n$-rotative} if it is $(n,a)$-rotative for some $a\in [0,n)$, and $T$ is called \emph{rotative} if it is $(n,a)$-rotative for some $n\geq 2$ and some $a\in [0,n)$.
\end{itemize}
We denote by $\Phi(C,n,a,k)$ the class of all $(n, a)$-rotative, $k$-Lipschitzian mappings $T$ from $C$ into itself.
\end{defn}
The following theorem shows that the condition of rotativeness is actually quite strong. It assures the existence of fixed points of nonexpansive mappings without boundedness or any special geometric structure on $C$.
\begin{thm}
(\cite{GK}) Let $C$ be a nonempty closed convex subset of a Banach space. If $T\in \Phi(C,n,a,1)$ for some $n\geq 2$ and $a\in [0,n)$, then $T$ has a fixed point.
\end{thm}
Moreover, Goebel and Koter \cite{GK1} show that rotativeness also assures the existence of fixed points in the case of $k$-Lipschitzian mappings with $k$ slightly greater than $1$ as stated in the next theorem.
\begin{thm}\label{liprotativehasfixedpoint}
(\cite{GK1}, see also \cite[p.\ 324-327]{KaKo}) Let $C$ be a nonempty closed convex subset of a Banach space $X$. For each $n\geq 2$ and $a\in [0,n)$, there exists $\gamma > 1$ such that if $k<\gamma$ and $T\in \Phi(C,n,a,k)$, then $T$ has a fixed point.
\end{thm}
Clearly, $\gamma$ in Theorem \ref{liprotativehasfixedpoint} depends on $X$, $n$, and $a$. Therefore it is natural to consider the function $\gamma(X,n,a)$ defined by
\begin{align*}
\gamma(X,n,a) = \inf \left\{\right.k\in [0,\infty) \mid &\text{there is a nonempty closed convex subset $C$ of}\\
&\text{ $X$ and $T\in \Phi(C,n,a,k)$ such that $\fix T = \emptyset$}\left.\right\}
\end{align*} 
According to Theorem \ref{liprotativehasfixedpoint}, it is known that $\gamma(X,n,a) > 1$ for any Banach space $X$, $n\geq 2$, and $a\in [0,n)$. By now upper bounds and lower bounds of $\gamma(X,n,a)$ have been obtained for some $X$, $n$, and $a$. However, the precise value of $\gamma(X,n,a)$ is completely unknown for any $X$, $n$, $a$. Here is a list of known results on $\gamma(X,n,a)$, and $\gamma(H,n,a)$ where $X$ is a Banach space, and $H$ is a Hilbert space. 
\begin{itemize}
\item[(i)] \cite{Goe, GK1, Ki} 
$$
\gamma(X,n,0) \geq \begin{cases}
2,&\text{for $n=2$},\\
\sqrt[n-1]{\frac{1}{n-2}\left(-1+\sqrt{n(n-1)-\frac{1}{n-1}}\right)}, &\text{for $n>2$}.
\end{cases}
$$
\item[(ii)] \cite{GoKi} For $a\in (0,2)$, 
\begin{align*}
\gamma(X,2,a) \geq \max&\left\{\frac{1}{2}\left(2-a+\sqrt{(2-a)^2+a^2}\right)\right.,\\
&\quad\left.\frac{1}{8}\left(a^2+4+\sqrt{(a^2+4)^2-64a+64}\right)\right\}.
\end{align*}
\item[(iii)] \cite{Kot} $\gamma(H,2,0) \geq \sqrt{\pi^2-3}$. 
\item[(iv)] \cite{Kom} $\gamma(H,2,a) \geq \sqrt{\frac{5}{a^2+1}}$ for any $a\in [0,2)$. 
\item[(v)] \cite{Goe1} 
$$
\gamma(C[0,1],2,a) \leq \frac{1}{a-1}
$$
where $a\in (1,2)$ and $C[0,1]$ is the space of real-valued continuous function on $[0,1]$. 
\item[(vi)] \cite{GP} 
\begin{align*}
\gamma(X,3,0) &\geq 1.3821,\quad\quad \gamma(X,4,0) \geq 1.2524,\\
\gamma(X,5,0) &\geq 1.1777, \;\text{and}\quad \gamma(X,6,0) \geq 1.1329.
\end{align*}
\item[(vii)] \cite{GaNa}
\begin{align*}
\gamma(H,3,0) &\geq 1.5549, \quad\quad\gamma(H,4,0) \geq 1.3267,\\
\gamma(H,5,0) &\geq 1.2152,\;\text{and}\quad \gamma(H,6,0) \geq 1.1562.
\end{align*}
\end{itemize} 
From the list of results given above, we see that even the largest lower bound of $\gamma(X,n,a)$ is smaller than $3$. So it is natural to ask the following questions:
\begin{itemize}
\item[Q1:] In what space $X$ is $\gamma(X,n,a)$ the largest? Is it a Hilbert space?
\item[Q2:] Can we find a Banach space $X$, $n\geq 2$, $a\in [0,n)$, and a function $T\in \Phi(C,n,a,k)$, for $k>3$ with $\fix T\neq \emptyset$?
\end{itemize}
Other questions concerning the function $\gamma$ are the following.
\begin{itemize}
\item[Q3:] For a Banach space $X$, what is a good estimation for $\gamma(X,n,0)$? \\
Is $\gamma(X,n,0) < \infty$? 
\item[Q4:] From the list (iii) given above, we know that $\gamma(C[0,1],2,a) \leq \frac{1}{a-1}$ for $a\in (1,2)$. But nothing is known for $a\in [0,1)$. \\
Is $\gamma(C[0,1],2,a) < \infty$ for some $a\in [0,1)$? \\
Is $\gamma(C[0,1],2,a) = \infty$ for some $a\in [0,1)$? 
\item[Q5:] Can we find a precise value of $\gamma(X,n,a)$ for some $X$, $n$, and $a$?
\item[Q6:] For a Banach space $X$ and $n\geq 2$, is $\gamma(X,n,\cdot):[0,n)\to (1,\infty]$ continuous?  
\end{itemize}
We will give an answer to some of these questions in the next section.
\section{Main Results}
Questions 1 to 6 remain open until today. We do not know a precise value of $\gamma(X,n,a)$ for any $X$, $n$, $a$. We obtain a small lower bound and do not know if it is closed to the best possible result. By restrict our attention to $X = \mathbb R$, the smallest nontrivial space, we obtain some answers to Q1 to Q6. We hope that this will shred some light to the current state of knowledge on rotative mappings and the function $\gamma$. Note that closed convex subsets of $\mathbb R$ are precisely closed intervals, and for a function $f$, we write $f^n$ for the $n$-fold composition $f\circ f\circ \cdots \circ f$. It is well known that every continuous selfmap on closed and bounded interval has a fixed point. We prove that every $2$-rotative continuous selfmap on closed (not necessity bounded) interval in $\mathbb R$ has a fixed point. 
\begin{thm}\label{mainthm}
Let $I$ be a nonempty closed interval in $\mathbb R$. Then every $2$-rotative continuous function $f:I\to I$ has a fixed point.
\end{thm} 
\begin{proof}
Let $f:I\to I$ be $2$-rotative and continuous. Then there exists $b\in [0,2)$ such that $\left|f^2(x)-x\right| \leq b\left|f(x)-x\right|$ for all $x\in I$. Suppose for a contradiction that $f$ has no fixed point. By the intermediate value property of continuous function on an interval, we have that either
$$
\text{$f(x) > x$ for all $x\in I$ or $f(x) < x$ for all $x\in I$.}
$$
In both cases, we note that
\begin{align*}
f^2(x) > f(x) > x&\quad\text{for all $x\in I$, or}\notag\\
f^2(x) < f(x) < x&\quad\text{for all $x\in I$}.
\end{align*} 
Therefore $\frac{f^2(x)-x}{f(x)-x} > 0$ for all $x\in I$. This implies that $b\in (1,2)$ and  
\begin{equation*}
1+\frac{f^2(x)-f(x)}{f(x)-x} = \frac{f^2(x)-x}{f(x)-x} = \left|\frac{f^2(x)-x}{f(x)-x}\right| \leq b
\end{equation*}
for each $x\in I$. That is for some $a\in (0,1)$
\begin{equation}\label{newintermediateeq3}
\left|\frac{f^2(x)-f(x)}{f(x)-x}\right| \leq a\quad\text{for all $x\in I$}.
\end{equation}
Now consider for each fixed $x\in I$. We will show that the sequence $(f^n(x))$ converges. For each $n\in \mathbb N$, we obtain by (\ref{newintermediateeq3}) that 
$$
\left|\frac{f^{n+1}(x)-f^n(x)}{f^n(x)-f^{n-1}(x)}\right| = \left|\frac{f^2\left(f^{n-1}(x)\right)-f\left(f^{n-1}(x)\right)}{f\left(f^{n-1}(x)\right)-f^{n-1}x}\right| \leq a.
$$
Therefore $\left|f^{n+1}(x)-f^n(x)\right| \leq a\left|f^n(x)-f^{n-1}(x)\right|$ for every $n\in\mathbb N$. This implies that $\left|f^{n+1}(x)-f^n(x)\right| \leq a^n\left|f(x)-x\right|$ for all $n\in\mathbb N$. Now for $m, n\in \mathbb N$ and $m>n$, we have
\begin{align}\label{roconthasfixedpointeq4}
\left|f^{m+1}(x)-f^n(x)\right| &\leq \left|f^{m+1}(x)-f^m(x)\right|+\left|f^{m}(x)-f^{m-1}(x)\right|+\cdots+\notag\\
&\quad\quad\left|f^{n+1}(x)-f^n(x)\right|\notag\\
&\leq \left(a^m+a^{m-1}+\cdots+a^n\right)\left|f(x)-x\right|\notag\\
&\leq \frac{a^n}{1-a}\left|f(x)-x\right|.
\end{align}
Since $a\in (0,1)$,
\begin{equation}\label{roconthasfixedpointeq5}
\lim_{n\to\infty}\frac{a^n}{1-a} = 0.
\end{equation}
By (\ref{roconthasfixedpointeq4}) and (\ref{roconthasfixedpointeq5}), we obtain that $(f^n(x))_{n\in\mathbb N}$ is a Cauchy sequence in $I$. Then $(f^n(x))$ converges to a point $x_0\in I$. Since $f$ is continuous, $f\left(f^n(x)\right)$ converges to $f(x_0)$. But $\left(f\left(f^n(x)\right)\right) = \left(f^{n+1}(x)\right)$ is a subsequence of $\left(f^n(x)\right)$, it converges to $x_0$. Therefore $f(x_0) = x_0$, a contradiction. This completes the proof.
\end{proof}
\begin{cor}\label{mainresultcor}
Let $I$ be a nonempty closed interval, $k\geq 0$, $a\in [0,2)$. If $T\in \Phi(I,2,a,k)$, then $T$ has a fixed point. In other words, every $2$-rotative $k$-Lipschitzian mapping on a closed interval has a fixed point.
\end{cor}
\begin{proof}
Follows immediately from Theorem \ref{mainthm}, since every Lipschitzian mapping is continuous. 
\end{proof}
Recall that by letting $\inf\emptyset = +\infty$ and $\sup \emptyset  = -\infty$, then we have the following result.
\begin{cor}\label{gammaequalinfty}
$\gamma(\mathbb R,2,a) = +\infty$, for every $a\in [0,2)$.
\end{cor}
\begin{proof}
By the definition of $\gamma$ and Corollary \ref{mainresultcor}, we obtain  
\begin{align*}
\gamma(\mathbb R,2,a) &= \inf\{k \in[0,\infty)\mid\; \text{there is a nonempty closed interval $I$ of $\mathbb R$ and }\\
&\quad\quad\quad\quad\quad\quad\quad\quad\quad\text{$T\in \Phi(I,2,a,k)$ such that $\fix T = \emptyset$}\}\\
&=\inf\emptyset\\
&= +\infty
\end{align*}
\end{proof}
Next we give a basic result concerning the properties of $\gamma(X,n,a)$.
\begin{prop}\label{supequalinf}
Let $X$ be a Banach space, $n\geq 2$, and $a\in [0,n)$. Then  
\begin{align*}
\gamma(X,n,a) = \sup\{k \in[0,\infty)\mid\; &\text{for every nonempty closed convex subset $C$ of $X$, }\\
&\text{if $T\in \Phi(C,n,a,k)$, then $T$ has a fixed point}\}.
\end{align*}
\end{prop}
\begin{proof}
Let 
\begin{align*}
A = \{k \in[0,\infty)\mid\; &\text{for every nonempty closed convex subset $C$ of $X$, }\\
&\text{if $T\in \Phi(C,n,a,k)$, then $T$ has a fixed point}\},\\
B = \{k \in[0,\infty)\mid\; &\text{there is a nonempty closed convex subset $C$ of $X$ and }\\
&\text{$T\in \Phi(C,n,a,k)$ such that $\fix T = \emptyset$}\}.
\end{align*}
By the definition, $\gamma(X,n,a) = \inf B$. So we only need to show that $\sup A = \inf B$. First observe that $a\leq b$ for every $a\in A$ and $b\in B$. Therefore $\sup A\leq \inf B$. If $B = \emptyset$, then $\inf B = +\infty$ and $\sup A = \sup[0,\infty) = +\infty$. So $\sup A = \inf B$. Assume that $B\neq \emptyset$. We will show that $\inf B-\sup A \leq \varepsilon$ for every $\varepsilon > 0$. Let $\varepsilon > 0$ and $\alpha = \sup A$. Then $\alpha +\varepsilon \notin A$. Then there exists a nonempty closed convex subset $C$ of $X$ and $T\in \Phi(C,n,a,\alpha+\varepsilon)$ with $\fix T = \emptyset$. By the definition of $B$, we see that $\alpha+\varepsilon\in B$. Therefore $\alpha+\varepsilon \geq \inf B$. Thus $\inf B-\sup A = \inf B-\alpha \leq \varepsilon$, as required. Hence $\inf B = \sup A$. 
\end{proof}
We end this section by giving an answer to some of Q1 to Q6 as follows:
\begin{itemize}
\item[Q1:] In what space $X$ is $\gamma(X,n,a)$ the largest? Is it a Hilbert space?
\item[Q5:] Can we find a precise value of $\gamma(X,n,a)$ for some $X$, $n$, and $a$?
\item[Answer:] We found that the precise value of $\gamma(\mathbb R,2,a)$ is $\infty$. But we do not know if $\gamma(H,2,a) = +\infty$ for every Hilbert space $H$.
\item[Q6:] For a Banach space $X$ and $n\geq 2$, is $\gamma(X,n,\cdot):[0,n)\to (1,\infty]$ continuous?  
\item[Answer:] If $X=\mathbb R$, then $\gamma(\mathbb R,2,\cdot)$ is a constant function, so it is continuous. However, we do not know the answer if $X\neq \mathbb R$.
\item[Q3:] For a Banach space $X$, what is a good estimation for $\gamma(X,n,0)$?\\
 Is $\gamma(X,n,0) < \infty$? 
\item[Answer:] If $X=\mathbb R$, the answer is no. We have $\gamma(\mathbb R,2,a) = +\infty$ for every $a\in [0,2)$. So in particular, $\gamma(\mathbb R,2,0) = +\infty$.
\end{itemize}
We will give an answer to Q2 in the next section.
\section{Examples of Rotative Mappings on $\mathbb R$}
In this section, we present some examples of rotative mappings. By replacing the condition $\left\|T^nx-x\right\|\leq a\left\|Tx-x\right\|$ by $d(T^nx,x) \leq ad(Tx,x)$, we see that the rotativeness can be used for a function defined on a metric space. In particular, we may talk about the rotativeness of mappings defined on a normed linear space. To keep the notation simple, we sometimes write $fx$ instead of $f(x)$, and $f^nx$ instead of $f^n(x)$.
\begin{exam}
It is not difficult to see that 
\begin{itemize}
\item[(i)] every contraction on a metric space is rotative,
\item[(ii)] every $n$-periodic map on a metric space is $(n,0)$-rotative, and
\item[(iii)] every rotation in $\mathbb R^m$ is rotative.
\end{itemize} 
\end{exam}
\begin{thm}\label{cx+dthm}
Let $X$ be a normed linear space. Let $n\geq 2$, $c\in \mathbb C$, $x_0\in X$, and $f:X\to X$ given by $f(x) = cx+x_0$. Then 
\begin{itemize}
\item[(i)] If $c\neq 1$, then $f$ is $n$-rotative if and only if $\left|\frac{c^n-1}{c-1}\right|<n$.
\item[(ii)] If $c = 1$, then $f$ is $n$-rotative if and only if $x_0=0$. 
\end{itemize}
\end{thm}
\begin{proof}
We have $fx = cx+x_0$, $f^2x = c^2x+cx_0+x_0$, and in general, $f^nx = c^nx+c^{n-1}x_0+c^{n-2}x_0+\cdots+cx_0+x_0$. So
\begin{align*}
f^nx - x = (c^n-1)x+(c^{n-1}+c^{n-2}+\cdots+c+1)x_0.
\end{align*}
If $c = 1$, then $\left\|f^nx - x\right\| = n\left\|x_0\right\|$ and $\left\|fx-x\right\| = \left\|x_0\right\|$. From this it is easy to see that $f$ is $n$-rotative if and only if $x_0=0$. If $c\neq 1$, then
\begin{align*}
f^nx - x &=(c^n-1)x+\frac{c^{n}-1}{c-1}x_0\\
&= \frac{c^n-1}{c-1}((c-1)x+x_0) = \frac{c^n-1}{c-1}(fx-x).
\end{align*}
Therefore $\left\|f^nx-x\right\| = \left|\frac{c^n-1}{c-1}\right|\left\|fx-x\right\|$ for all $x\in X$. Hence $f$ is $n$-rotative if and only if $\left|\frac{c^n-1}{c-1}\right|< n$. This completes the proof.
\end{proof}
\begin{cor}\label{cx+dcor}
Let $X$ be a normed linear space. Let $c\in \mathbb R-\{1\}$, $x_0\in X$, $f:X\to X$ given by $f(x) = cx+x_0$. Then 
\begin{itemize}
\item[(i)] $f$ is $2$-rotative if and only if $-3<c<1$.
\item[(ii)] $f$ is $3$-rotative if and only if $-2<c<1$.
\end{itemize}
\end{cor}
\begin{proof}
Immediately obtained from Theorem \ref{cx+dthm}.
\end{proof}
\begin{thm}\label{fcon2rothm}
Let $c_1, c_2, b_1, b_2\in \mathbb R$ be such that $c_1<c_2<b_1<b_2$. Let $f:\mathbb R\to\mathbb R$ be given by
$$
f(x) = \begin{cases}
c_1, &x\leq b_1;\\
c(x-b_1)+c_1, &b_1<x < b_2;\\
c_2, &x\geq b_2,
\end{cases}
$$
where $c = \frac{c_2-c_1}{b_2-b_1}$. Then 
\begin{itemize}
\item[(i)] $f$ is continuous and  $\fix f = \{c_1\}$.
\item[(ii)] $f$ is $n$-rotative if and only if $b_1>\frac{nc_2-c_1}{n-1}$.
\end{itemize}
\end{thm}
\begin{proof}
It is easy to see that $f$ is continuous and $\fix f = \{c_1\}$. Next we will prove that $f$ is $n$-rotative if and only if $b_1>\frac{nc_2-c_1}{n-1}$. Let $x\in\mathbb R$ and $n\geq 2$.\\
If $x\leq b_1$, then $f(x) = c_1$ and $f^n(x) = c_1$, and so
\begin{equation}\label{fcont2roeq1}
\left|x-f^n(x)\right| = \left|x-c_1\right| = \left|x-f(x)\right|
\end{equation}
If $x\geq b_2$, then $f(x) = c_2$ and $f^n(x) = c_1$, and thus
\begin{equation}\label{fcont2roeq2}
\frac{\left|x-f^n(x)\right|}{\left|x-f(x)\right|} = \frac{x-c_1}{x-c_2} = 1+\frac{c_2-c_1}{x-c_2}
\end{equation}
If $b_1<x<b_2$, then $f(x) < c_2$, $f^n(x) = c_1$, and therefore 
\begin{equation}\label{fcont2roeq3}
\frac{\left|x-f^n(x)\right|}{\left|x-f(x)\right|} = \frac{x-c_1}{x-f(x)} < \frac{x-c_1}{x-c_2} = 1+\frac{c_2-c_1}{x-c_2}
\end{equation}
In conclusion, we have for every $n\geq 2$
$$
\frac{\left|x-f^n(x)\right|}{\left|x-f(x)\right|} \leq \begin{cases}
1, &\text{if $x\leq b_1$ and $x\neq c_1$};\\
1+\frac{c_2-c_1}{x-c_2}, &\text{if $x>b_1$}.
\end{cases}
$$
From this, we see that
\begin{equation}\label{lipsmapeq100}
\sup \left\{\left|\frac{x-f^n(x)}{x-f(x)}\right|\mid x\in \mathbb R-\{c_1\}\right\} = \sup \left\{1+\frac{c_2-c_1}{x-c_2}\mid x>b_1\right\} = 1+\frac{c_2-c_1}{b_1-c_2}.
\end{equation}
Now assume that $b_1>\frac{nc_2-c_1}{n-1}$. Then we let $a = 1+\frac{c_2-c_1}{b_1-c_2}$ so that $a\in (1,n)$ and by (\ref{lipsmapeq100}),
\begin{equation}\label{lipsmapeq101}
\text{$\left|x-f^n(x)\right| \leq a\left|x-f(x)\right|$\quad\quad for all $x\in\mathbb R-\{c_1\}$}.
\end{equation}
Since $\fix f = \{c_1\}$, the inequality in (\ref{lipsmapeq101}) also holds for $x=c_1$. Therefore (\ref{lipsmapeq101}) holds for every $x\in \mathbb R$. This shows that $f$ is $n$-rotative. \\
Conversely, assume that $f$ is $n$-rotative. Then there exists $a\in (0,n)$ such that 
$$
\text{$\left|x-f^n(x)\right| \leq a\left|x-f(x)\right|$\quad\quad for all $x\in\mathbb R$}.
$$
Then $a\geq \frac{\left|x-f^n(x)\right|}{\left|x-f(x)\right|}$ for every $x\in\mathbb R-\{c_1\}$. So we obtain by (\ref{lipsmapeq100}) that 
$$
a\geq 1+\frac{c_2-c_1}{b_1-c_2}. 
$$
Since $a<n$, $1+\frac{c_2-c_1}{b_1-c_2} < n$. Hence $b_1 > \frac{nc_2-c_1}{n-1}$, as required. This completes the proof.
\end{proof}
\begin{exam}
Let $f, g:\mathbb R\to \mathbb R$ be given by 
$$
f(x) = \begin{cases}
1, &x\leq 12;\\
5(x-12)+1, &12 < x < 13;\\
6, & x\geq 13,
\end{cases}
\quad g(x) = \begin{cases}
1, &x\leq \frac{199}{99};\\
x-\frac{100}{99}, &\frac{199}{99} < x < \frac{298}{99};\\
2, & x\geq \frac{298}{99}.
\end{cases}
$$
The function $f$ corresponds to the case $c_1 = 1$, $c_2 = 6$, $b_1 = 12$, $b_2 = 13$ in Theorem \ref{fcon2rothm}. Since $b_1 > \frac{nc_2-c_1}{n-1}$ for every $n\geq 2$, $f$ is $n$-rotative for every $n\geq 2$. The function $g$ corresponds to the case $c_1 = 1$, $c_2 = 2$, $b_1 = 2+\frac{1}{99}$, $b_2 = 3+\frac{1}{99}$. It is easy to check that $b_1 > \frac{nc_2-c_1}{n-1}$ if and only if $n>100$. Therefore $g$ is not $n$-rotative for $n\in [2,100]$ and $g$ is $n$-rotative for $n\geq 101$.
\end{exam}
Next we show that there exists a function $f:\mathbb R\to \mathbb R$ which is $n$-rotative for infinitely many $n$ but is not $m$-rotative for infinitely many $m$.
\begin{exam}
Let $b\in \mathbb Q$, $c\in \mathbb Q^c$, and let $f:\mathbb R\to\mathbb R$ be given by 
$$
f(x) = \begin{cases}
c, &x\in \mathbb Q;\\
b, &x\in \mathbb Q^c.
\end{cases}
$$ 
It is easy to see that 
\begin{itemize}
\item[] if $x\in \mathbb Q$, then $f^n(x) = \begin{cases}
c, &\text{if $n$ is odd};\\
b, &\text{if $n$ is even},
\end{cases}$
and 
\item[] if $x\in \mathbb Q^c$, then $f^n(x) = \begin{cases}
b, &\text{if $n$ is odd};\\
c, &\text{if $n$ is even}.
\end{cases}$
\end{itemize}
So if $n$ is odd, then $x-f^n(x) = x-f(x)$ for all $x\in \mathbb R$. Therefore $f$ is $(n,1)$-rotative for every odd integer $n\geq 3$. If $m\geq 2$ is even, let $x = \frac{mb-c}{m-1}$, so that $x\in \mathbb Q^c$ and $\frac{x-f^m(x)}{x-f(x)} = \frac{x-c}{x-b} = m$. Therefore $f$ is not $m$-rotative for any even integer $m\geq 2$.
\end{exam}
\begin{cor}\label{corexamthereexistm}
Let $M>0$. Then there exists a function $f:\mathbb R\to\mathbb R$ such that $f$ is $k$-Lipschitzian, $2$-rotative, $k = M$, and $\fix f\neq \emptyset$.
\end{cor}
\begin{proof}
Let $c_1 = 1$, $c_2 = 2$, $b_1 = 3$, $b_2 = 3+\frac{1}{M}$, and let $c\in\mathbb R$, $f:\mathbb R\to\mathbb R$ be defined as in Theorem \ref{fcon2rothm}. Then $|f(x)-f(y)| \leq c|x-y|$ for all $x, y\in\mathbb R$, and $|f(b_2)-f(b_1)| = c|b_2-b_1|$. So $f$ is $k$-Lipschitzian where $k = c = \frac{c_2-c_1}{b_2-b_1} = M$, and $\fix f = \{1\}$. Since $b_2 > 3 = 2c_2-c_1$, we see that $f$ is 2--rotative by Theorem \ref{fcon2rothm}.
\end{proof}
Now we can give an answer to Q2.
\begin{itemize}
\item[Q2:] Can we find a function $T\in \Phi(C,n,a,k)$, for $k > 3$ with $\fix T\neq \emptyset$?
\item[Answer:] By Corollary \ref{corexamthereexistm}, there exists a function $T\in \Phi(\mathbb R,2,a,k)$ with $\fix T\neq \emptyset$ and $k$ is arbitrarily large.
\end{itemize}
\noindent \textbf{Competing Interests}\\[0.2cm]
The authors declare that they have no competing interests.\\[0.3cm]
\noindent \textbf{Authors' contributions}\\[0.2cm]
All authors contributed significantly in writing this paper. All authors read and approved this final manuscript.\\[0.3cm]
\noindent \textbf{Author details}\\[0.2cm]
$^a$ Department of Mathematics and Computer Science, Faculty of Science, Chulalongkorn University, Bangkok, 10330, Thailand.\\[0.3cm]
\noindent\textbf{Acknowledgment} \\[0.2cm]
We would like to thank Professor Goebel for mentioning this problem to us which leads to the publication of this article.   


\begin{thebibliography}{99}
\bibitem {GaNa} V. P. Garc\'ia and H. F. Nathansky, \emph{Fixed points of periodic mappings in Hilbert spaces}, Annales Universitatis Mariae curie-Sk\l odowska Lublin-Polonia, Vol. LXIV (2) (2010), 37--48.
\bibitem {Goe1} K. Goebel, \emph{On the minimal displacement of points under Lipschitzian mappings}. Pacific J. Math. \textbf{45} (1973), 151--163.
\bibitem {Goe} K. Goebel, \emph{Convexity of balls and fixed points theorems for mappings with nonexpansive square}. Compositio Math. \textbf{22} (1970), 269--274. 
\bibitem{GoKi} K. Goebel and W. A. Kirk, \textit{Topics in Metric Fixed Point Theory}, Cambridge University Press, 1990.
\bibitem {GK} K. Goebel and M. Koter, \emph{A remark on nonexpansive mappings}, Canad. Math. Bull. \textbf{24} (1981), 113--115.
\bibitem {GK1} K. Goebel and M. Koter, \emph{Fixed points of rotative Lipschitzian mappings}, Rend. Sem. Mat. Fis. di Milano \textbf{51} (1981)
, 145--156.
\bibitem {GP} J. G\'ornicki and K. Pupka, \emph{Fixed points of rotative mappings in Banach spaces}, J. Nonlinear Convex Anal. \textbf{6}(2), (2005), 217--233.
\bibitem{KaKo}  W. Kaczor and M. Koter, \textit{Rotative mappings and mappings with constant displacement}, Handbook of Metric Fixed Point Theory, W. A. Kirk \& B. Sims (editors), Kluwer Academic Publishers, 2001, 323--337.
\bibitem {Ki} W. A. Kirk, \emph{A fixed theorem for mappings with a nonexpansive iterate}, Proc. Amer. Math. Soc. \textbf{29} (1971), 294--298. 
\bibitem {Kom} T. Komorowski, \emph{Selected topics on lipschitzian mappings}, (in Polish) Thesis, Univ. Maria Curie-sk\l odowska, 1987.
\bibitem {Kot} M. Koter, \emph{Fixed points of lipschitzian 2-rotative mappings}, Boll. Un. Mat. Ital. Ser. VI, \textbf{5} (1986), 321--339.
\end{thebibliography}
\end{document}